\documentclass[11pt]{amsart}
\usepackage{amsmath}
\newtheorem{thm}{Theorem}[section]
\newtheorem{cor}[thm]{Corollary}
\newtheorem{lem}[thm]{Lemma}
\newtheorem{prop}[thm]{Proposition}
\theoremstyle{definition}

\theoremstyle{remark}

\numberwithin{equation}{section}

\newcommand{\C}{\mathbb{C}}

\begin{document}
\title[$k$-Transitivity abelian  semigroup of affine maps ]{ Hypercyclicity and $k$-Transitivity ($k\geq 2$) \\ for abelian  semigroup of
affine maps on $\mathbb{C}^{n}$  }

\author{Yahya N'dao}

 \address{Yahya N'dao, University of Moncton, Department of mathematics and statistics, Canada}
 \email{yahiandao@yahoo.fr ; yahiandao@voila.fr}



\begin{abstract}
In this paper, we prove that the minimal number of affine maps on
$\mathbb{C}^{n}$, required to form a hypercyclic abelian semigroup
on $\mathbb{C}^{n}$ is $n+1$. We also prove that the action of any
abelian group finitely generated by affine maps on
$\mathbb{C}^{n}$, is never k-transitive for $k\geq 2$.
\end{abstract}
\maketitle

\section{\bf Introduction }

Let $M_{n}(\mathbb{C})$ be the set of all square matrices of order
$n\geq 1$  with entries in  $\mathbb{C}$ and $GL(n, \ \mathbb{C})$
be the group of all invertible matrices of $M_{n}(\mathbb{C})$. A
map $f: \ \mathbb{C}^{n}\longrightarrow \mathbb{C}^{n}$  is called
an affine map if there exist $A\in M_{n}(\mathbb{C})$  and  $a\in
\mathbb{C}^{n}$  such that  $f(x)= Ax+a$,  $x\in \mathbb{C}^{n}$.
We denote  $f= (A,a)$, we call  $A$ the \textit{linear part} of
$f$. The map $f$ is invertible if $A\in GL(n, \mathbb{C})$. Denote
by $MA(n, \ \mathbb{C})$  the vector space of all affine maps on
$\mathbb{C}^{n}$ and $GA(n, \ \mathbb{C})$ the group of all
invertible affine maps of $MA(n, \mathbb{C})$. Let  $\mathcal{G}$
be an abelian affine sub-semigroup of $MA(n, \ \mathbb{C})$. For a
vector  $x\in \mathbb{C}^{n}$,  we consider the orbit of
$\mathcal{G}$  through  $x$: $O_{\mathcal{G}}(x) = \{f(x): \ f\in
\mathcal{G}\} \subset \mathbb{\mathbb{C}}^{n}$. Denote by
$\overline{E}$ (resp. $\overset{\circ}{E}$) the closure (resp.
interior) of $E$.  The orbit $O_{\mathcal{G}}(x)\subset
\mathbb{C}^{n}$ is dense (resp. locally dense) in $\mathbb{C}^{n}$
if $\overline{O_{\mathcal{G}}(x)} = \mathbb{C}^{n}$ (resp.
$\overset{\circ}{ \overline{O_{\mathcal{G}}(x)}} \neq \emptyset$).
The semigroup $G$ is called \emph{hypercyclic} (or also
\emph{topologically transitive}) (resp. \emph{locally
hypercyclic}) if there exists a point $x\in\mathbb{C}^{n}$ such
that $O_{\mathcal{G}}(x)$ is dense (resp. locally dense) in
$\mathbb{C}^{n}$. For an account of results and bibliography on
hypercyclicity, we refer to the book \cite{bm} by Bayart and
Matheron.

\medskip

On the other part, let $k\geq 1$ be an integer, denote by
$(\mathbb{C}^{n})^{k}$ the $k$-fold Cartesian product of
$\mathbb{C}^{n}$. For every $u = (x_{1}, \dots, x_{k})\in
(\mathbb{C}^{n})^{k}$, the orbit of $u$ under the action of
$\mathcal{G}$ on $(\mathbb{C}^{n})^{k}$ is denoted
$$O^{k}_{\mathcal{G}}(u) = \{(f(x_{1}),\dots ,f(x_{k})) : f \in
\mathcal{G}\}.$$ When $k=1$, $O^{k}_{\mathcal{G}}(u)=
O_{\mathcal{G}}(u)$. We say that the action of $\mathcal{G}$ on
$\mathbb{C}^{n}$ is $k$-transitive if, the induced action of
$\mathcal{G}$ on $(\mathbb{C}^{n})^{k}$ is hypercyclic, this is
equivalent to that for some $u\in (\mathbb{C}^{n})^{k}$,
$\overline{O^{k}_{\mathcal{G}}(u)}=(\mathbb{C}^{n})^{k}$. A
$2$-transitive action is also called weak topological mixing and
$1$-transitive means hypercyclic.

In this paper, we give a  generalization of the paper \cite{AA} in
which A.Ayadi proved that any abelian semigroup finitely generated
by matrices on $\mathbb{C}^{n}$ is never $k$-transitive for every
$k\geq 2$.

\medskip

Our principal results are the following:\\

\begin{thm}\label{T:00} For every $n \geq 1$, any abelian semigroup generated by n
affine maps on $\mathbb{C}^{n}$, is not locally hypercyclic.
\end{thm}
\medskip

 \begin{thm}\label{T:1} Let $\mathcal{G}$ be an abelian sub-semigroup
of $MA(n,  \mathbb{C})$ and generated by $p$ affine maps, $p\geq
1$. Then $\mathcal{G}$ is not k-transitive for every $k\geq 2$.
\end{thm}
\bigskip

\section{\bf Notations and some results for abelian linear semigroup}

 Let $n\in\mathbb{N}_{0}$  be fixed, denote by:
\\
\textbullet \ $\mathbb{C}^{*}= \mathbb{C}\backslash\{0\}$,
$\mathbb{R}^{*}= \mathbb{R}\backslash\{0\}$  and $\mathbb{N}_{0}=
\mathbb{N}\backslash\{0\}$.
\
\\
\textbullet \;  $\mathcal{B}_{0} = (e_{1},\dots,e_{n+1})$ the
canonical basis of $\mathbb{C}^{n+1}$  and   $I_{n+1}$  the
identity matrix of $GL(n+1,\mathbb{C})$.
\
\\
For each $m=1,2,\dots, n+1$, denote by:
\\
\textbullet \; $\mathbb{T}_{m}(\mathbb{C})$ the set of matrices
over $\mathbb{C}$ of the form $$\left[\begin{array}{cccc}
    \mu & \ &  \ & 0 \\
  a_{2,1} & \mu &  \ &  \\\
  \vdots &  \ddots & \ddots & \ \\
  a_{m,1} & \dots & a_{m,m-1} & \mu\\
\end{array}\right]\ \ \  \label{eq1}$$
\
\\
\textbullet \; $\mathbb{T}_{m}^{\ast}(\mathbb{C})$  the group of
matrices of the form (~\ref{eq1}) with $\mu\neq 0$.
\
\\
Let  $r\in \mathbb{N}$ and $\eta
=(n_{1},\dots,n_{r})\in\mathbb{N}^{r}_{0}$ such that $n_{1}+\dots
+ n_{r}=n+1. $ In particular, $r\leq n+1$. Write
\
\\
\textbullet \; \ $\mathcal{K}_{\eta,r}(\mathbb{C}): =
\mathbb{T}_{n_{1}}(\mathbb{C})\oplus\dots \oplus
\mathbb{T}_{n_{r}}(\mathbb{C}).$ In particular if  $r=1$, then
$\mathcal{K}_{\eta,1}(\mathbb{C}) = \mathbb{T}_{n+1}(\mathbb{C})$
and  $\eta=(n+1)$.
\
\\
\textbullet \; $\mathcal{K}^{*}_{\eta,r}(\mathbb{C}): =
\mathcal{K}_{\eta,r}(\mathbb{C})\cap \textrm{GL}(n+1, \
\mathbb{C})$.\
\
\\
\textbullet \; $u_{0} = (e_{1,1},\dots,e_{r,1})\in
\mathbb{C}^{n+1}$ where
 $e_{k,1} = (1,0,\dots,
0)\in \mathbb{C}^{n_{k}}$, for $k=1,\dots, r$. So
$u_{0}\in\{1\}\times\mathbb{C}^{n}$.
\\
\textbullet \; $p_{2}:\mathbb{C}\times
\mathbb{C}^{n}\longrightarrow\mathbb{C}^{n}$ the second projection
defined by $p_{2}(x_{1},\dots,x_{n+1})=(x_{2},\dots,x_{n+1})$.\
\\
 \textbullet \; $e^{(k)} = (e^{(k)}_{1},\dots,  e^{(k)}_{r})\in \mathbb{C}^{n+1}$ where $$e^{(k)}_{j} = \left\{\begin{array}{c}
  0\in \mathbb{C}^{n_{j}}\ \ \mathrm{if}\ \ j\neq k \\
  e_{k,1}\ \ \ \ \ \ \ \ \mathrm{if}\ \ j = k  \\
\end{array}\right. \ \ \ \ \ \  for \ \ every\ \ 1\leq j,\ k\leq r.$$
\\
\textbullet \; $\textrm{exp} :\ \mathbb{M}_{n+1}(\mathbb{C})
\longrightarrow\textrm{GL}(n+1, \mathbb{C})$  is the matrix
exponential map; set $\textrm{exp}(M) = e^{M}$, $M\in
M_{n+1}(\mathbb{C})$.\
\\
\textbullet \; Define the map \  \  \ $\Phi\ :\ GA(n,\
\mathbb{C})\ \longrightarrow\
 GL(n+1,\ \mathbb{C})$ \\
  $$f =(A,a) \ \longmapsto\ \begin{bmatrix}
                     1  & 0 \\
                     a & A
                 \end{bmatrix}$$
We have the following composition formula
 $$\begin{bmatrix}
                     1  & 0 \\
                     a & A
                 \end{bmatrix}\begin{bmatrix}
                     1  & 0 \\
                     b & B
                 \end{bmatrix} = \begin{bmatrix}
                     1  & 0 \\
                     Ab+a & AB
                 \end{bmatrix}.$$
 Then  $\Phi$  is an injective homomorphism of groups.\ Write
  \
\\
\textbullet \; $G:=\Phi(\mathcal{G})$, it is an abelian
sub-semigroup of $GL(n+1, \mathbb{C})$.
\\
\textbullet \;  Define the map \ \ \ $\Psi\ :\ MA(n,\ \mathbb{C})\
\longrightarrow\
 M_{n+1}(\mathbb{C})$
$$f =(A,a) \ \longmapsto\ \begin{bmatrix}
                     0  & 0 \\
                     a & A
                 \end{bmatrix}$$
We can see that   $\Psi$  is injective and linear. Hence
$\Psi(MA(n, \mathbb{C}))$ is a vector subspace of
$M_{n+1}(\mathbb{C})$. We prove (see Lemma~\ref{L:10000001}) that
$\Phi$ and $\Psi$ are related by the following property
$$\textrm{exp}(\Psi(MA(n, \mathbb{C})))=\Phi(GA(n, \mathbb{C})).$$

Let consider the normal form of $\mathcal{G}$: By Proposition
~\ref{p:2}, there exists a $P\in \Phi(\textrm{GA}(n, \mathbb{C}))$
and a partition $\eta$ of $(n+1)$ such that $G'=P^{-1}GP\subset
\mathcal{K}^{*}_{\eta,r}(\mathbb{C})\cap\Phi(MA(n,\mathbb{C}))$.
For such a choice of matrix $P$,  we let
\
\\
\textbullet \; $v_{0} = Pu_{0}$. So $v_{0}\in
\{1\}\times\mathbb{C}^{n}$, since
  $P\in\Phi(GA(n, \mathbb{C}))$.
 \\
 \textbullet \; $w_{0} = p_{2}(v_{0})\in\mathbb{C}^{n}$. We have $v_{0}=(1, w_{0})$.\
 \\
\textbullet \; $\varphi=\Phi^{-1}(P)\in MA(n,\mathbb{C})$.
 \\
\textbullet\  $\mathrm{g} = \textrm{exp}^{-1}(G)\cap \left(
P(\mathcal{K}_{\eta,r}(\mathbb{C}))P^{-1}\right)$. If $G\subset
\mathcal{K}^{*}_{\eta,r}(\mathbb{C})$, we have $P=I_{n+1}$ and
$\mathrm{g} = \textrm{exp}^{-1}(G)\cap
\mathcal{K}_{\eta,r}(\mathbb{C})$.
\
\\
\textbullet\  $\mathrm{g}^{1} = \mathrm{g}\cap
\Psi(MA(n,\mathbb{C}))$. It is an additive sub-semigroup of
$M_{n+1}(\mathbb{C})$ (because by  Lemma ~\ref{LL:002},
$\mathrm{g}$ is an additive sub-semigroup of
$M_{n+1}(\mathbb{C})$).
\
\\
\textbullet\ $\mathcal{G}^{*}=\mathcal{G}\cap GA(n, \mathbb{C})$.
\
\\
\textbullet\  $\mathrm{g}^{1}_{u} = \{Bu: \ B\in
\mathrm{g}^{1}\}\subset \mathbb{C}^{n+1}, \ \
u\in\mathbb{C}^{n+1}.$
\
\\
\textbullet\   $\mathfrak{q} = \Psi^{-1}(\mathrm{g}^{1})\subset
MA(n, \mathbb{C})$.

 Then $\mathfrak{q}$ is an additive sub-semigroup of $MA(n,\mathbb{C})$ and we have $\Psi(\mathfrak{q})=\mathrm{g}^{1}$.
  By Corollary~\ref{r:1}, we have \ $exp(\Psi(\mathfrak{q}))=\Phi(\mathcal{G}^{*}).$
\\
\textbullet\  $\mathfrak{q}_{v} =\{f(v),\ \ f\in
\mathfrak{q}\}\subset \mathbb{C}^{n}, \ \ v\in\mathbb{C}^{n}.$
\
\\

For groups of affine maps on  $\mathbb{K}^{n}$
($\mathbb{K}=\mathbb{R}$ or $\mathbb{C}$), their dynamics were
recently initiated for some classes in different point of view,
(see for instance, \cite{Ja}, \cite{Ku}, \cite{be}, \cite{AMY}).
The purpose here is to give analogous results as for linear
abelian sub-semigroup of $M_{n}(\mathbb{C})$.\
 \\
 \\

\ \\ In \cite{AAFF}, the authors proved the following Proposition:

 \begin{prop}\label{p:2}$($\cite{AAFF}, Proposition 2.1$)$ Let  $\mathcal{G}$ be an abelian subgroup of
$GA(n,\mathbb{C})$ and  $G=\Phi(\mathcal{G})$. Then there exists
$P\in \Phi(GA(n,\mathbb{C}))$ such that $P^{-1}GP$ is a subgroup
of $\mathcal{K}^{*}_{\eta,r}(\mathbb{C})\cap\Phi(GA(n,
\mathbb{C}))$,
 for some $r\leq n+1$ and $\eta=(n_{1},\dots,n_{r})\in\mathbb{N}_{0}^{r}$.
\end{prop}
\medskip

This proposition can be generalized for any abelian affine
semigroup as follow:
\begin{prop}\label{p:0011}   Let  $\mathcal{G}$ be an abelian sub-semigroup of
$MA(n,\mathbb{C})$ and  $G=\Phi(\mathcal{G})$. Then there exists
$P\in \Phi(GA(n,\mathbb{C}))$ such that $P^{-1}GP$ is a subgroup
of $\mathcal{K}^{*}_{\eta,r}(\mathbb{C})\cap\Phi(GA(n,
\mathbb{C}))$,
 for some $r\leq n+1$ and $\eta=(n_{1},\dots,n_{r})\in\mathbb{N}_{0}^{r}$.
\end{prop}

The same proof of Proposition~\ref{p:2} remained valid for
Proposition~\ref{p:0011}.

\medskip
 \ \\ \\ For such matrix $P$ define $v_{0}=Pu_{0}$. Let $L$
be an abelian subgroup of $\mathcal{K}^{*}_{\eta,r}(\mathbb{C})$.
denote by:\ \\ -
$V=\underset{k=1}{\overset{r}{\prod}}\mathbb{C}^{*}\times
\mathbb{C}^{n_{k}-1}$. One has $\mathbb{C}^{n}\backslash
V=\underset{k=1}{\overset{r}{\bigcup}}H_{k}$, where
$$H_{k}=\left\{u=[u_{1},\dots,u_{r}]^{T},\ \ u_{k}\in \{0\}\times
\mathbb{C}^{n_{k}-1},\ u_{j}\in\mathbb{C}^{n_{j}},\ j\neq k
\right\}.$$ See that each $H_{k}$ is a $L$-invariant vector space
of dimension $n-1$.
\medskip

\begin{lem}\label{L:10000001} $($\cite{AAFF}, Lemma 2.8$)$ $exp(\Psi(MA(n, \mathbb{C}))=GA(n, \mathbb{C})$.
\end{lem}

Denote by $G^{*}=G\cap GL(n+1, \mathbb{C})$.

\begin{lem}\label{L:21013}$($ ~\cite{aA-Hm12}$)$ One has
$exp(\mathrm{g})=G^{*}$.
\end{lem}
\medskip

\begin{cor}\label{C:21012} Let $G=\Phi(\mathcal{G})$. We have $\mathrm{g}= \mathrm{g}^{1}+2i\pi \mathbb{Z}I_{n+1}$.\
\end{cor}

\begin{proof} Let $N\in \mathrm{g}$. By Lemma ~\ref{L:21013}, $exp(N)\in G^{*}\subset \Phi(GA(n, \mathbb{C}))$.
Then by Lemma ~\ref{L:4}, there exists $k\in \mathbb{Z}$ such that
$N'=N-2ik\pi I_{n+1}\in\Psi(MA(n, \mathbb{C}))$. As
$e^{N'}=e^{N}\in G^{*}$ and $N'\in
P\mathcal{K}_{\eta,r}(\mathbb{C})P^{-1}$ then $N'\in
\mathrm{g}\cap \Psi(MA(n, \mathbb{C}))=\mathrm{g}^{1}$. Hence
$\mathrm{g}\subset \mathrm{g}^{1}+2i\pi\mathbb{Z}I_{n+1}$.
 Conversely, as $\mathrm{g}^{1}+2i\pi \mathbb{Z}I_{n+1}\subset P\mathcal{K}_{\eta,r}(\mathbb{C})P^{-1}$ and
 $exp(\mathrm{g}^{1}+2i\pi \mathbb{Z}I_{n+1})=exp(\mathrm{g}^{1})\subset G^{*}$, hence $\mathrm{g}^{1}+2i\pi\mathbb{Z}I_{n+1}\subset \mathrm{g}$.
\end{proof}
\medskip

\begin{cor}\label{r:1} We have $exp(\Psi(\mathfrak{q}))=\Phi(\mathcal{G}^{*})$.
\end{cor}
\medskip

\begin{proof}
By Lemmas ~\ref{L:21013} and ~\ref{C:21012}, We have
$G^{*}=exp(\mathrm{g})=exp(\mathrm{g}^{1}+2i\pi
\mathbb{Z}I_{n+1})=exp(\mathrm{g}^{1})$. Since
$\mathrm{g}^{1}=\Psi(\mathfrak{q})$, we get
$exp(\Psi(\mathfrak{q}))=\Phi(\mathcal{G}^{*})$.
\end{proof}
\medskip

\section{{\bf Proof of Theorem~\ref{T:00}}}

Let $\widetilde{G}$ be the semi-group generated by $G$ and
$\mathbb{C}I_{n+1}=\{\lambda I_{n+1}:\  \  \ \lambda\in \mathbb{C}
\}$. Then $\widetilde{G}$ is an abelian sub-semigroup of
$M_{n+1}(\mathbb{C})$. By Proposition~\ref{p:2}, there exists
$P\in\Phi(GA(n, \mathbb{C}))$ such that $P^{-1}GP$ is a subgroup
of $\mathcal{K}_{\eta,r}(\mathbb{C})$ for some $r\leq n+1$ and
$\eta=(n_{1},\dots,n_{r})\in\mathbb{N}_{0}^{r}$ and this also
implies that $P^{-1}\widetilde{G}P$ is a sub-semigroup of
$\mathcal{K}_{\eta,r}(\mathbb{C})$. Set
$\widetilde{\mathrm{g}}=exp^{-1}(\widetilde{G})\cap
(P\mathcal{K}_{\eta,r}(\mathbb{C})P^{-1})$ \ and \
 $\widetilde{\mathrm{g}}_{v_{0}}=\{Bv_{0}\ : \ B\in
\widetilde{\mathrm{g}}\}$.\ Then we have the following theorem,
applied to $\widetilde{G}$:
\medskip

\begin{thm}\label{T:5}$($\cite{aA-Hm12},\ Theorem 1.1$)$  Under the notations above, the following properties are equivalent:
\begin{itemize}
  \item [(i)] $\widetilde{G}$ has a dense orbit in
  $\mathbb{C}^{n+1}$.
  \item [(ii)] the orbit $O_{\widetilde{G}}(v_{0})$ is dense in
$\mathbb{C}^{n+1}$.
  \item [(iii)] $\widetilde{\mathrm{g}}_{v_{0}}$ is an additive subgroup dense in
$\mathbb{C}^{n+1}$.
\end{itemize}
\end{thm}
\medskip

\begin{lem}\label{LL:002}$($\cite{aAh-M05},\ Lemma 4.1$)$ The sets $\mathrm{g}$ and \ $\widetilde{\mathrm{g}}$ are additive subgroups of
$M_{n+1}(\mathbb{C})$. In particular, $\mathrm{g}_{v_{0}}$ and
$\widetilde{\mathrm{g}}_{v_{0}}$ are additive sub-semigroups of
$\mathbb{C}^{n+1}$.
\end{lem}
\medskip

Recall that $\mathrm{g}^{1}=\mathrm{g}\cap \Psi(MA(n,
\mathbb{C}))$ and $\mathfrak{q}=\Psi^{-1}(\mathrm{g}^{1})\subset
MA(n, \mathbb{C})$.
\medskip

\begin{lem}\label{L:01234} Under the notations above, one has:\
\begin{itemize}
  \item [(i)] $\widetilde{\mathrm{g}}=\mathrm{g}^{1}+\mathbb{C}I_{n+1}$.\
  \item [(ii)] $\{0\}\times\mathfrak{q}_{w_{0}}=\mathrm{g}^{1}_{v_{0}}$.
 \end{itemize}
\end{lem}
\medskip

\begin{proof} $(i)$ Let $B\in \widetilde{\mathrm{g}}$, then $e^{B}\in \widetilde{G}$. One can write $e^{B}=\lambda A$ for some
 $\lambda\in \mathbb{C}^{*}$ and $A\in G$. Let $\mu\in \mathbb{C}$ such that $e^{\mu}=\lambda$, then $e^{B-\mu I_{n+1}}=A$.
 Since $B-\mu I_{n+1}\in
  P\mathcal{K}_{\eta,r}(\mathbb{C})P^{-1}$, so $B-\mu I_{n+1}\in exp^{-1}(G)\cap P\mathcal{K}_{\eta,r}(\mathbb{C})P^{-1}=\mathrm{g}$.
   By Corollary~\ref{C:21012}, there exists $k\in \mathbb{Z}$ such that
   $B':=B-\mu I_{n+1}+2ik\pi I_{n+1}\in \mathrm{g}^{1}$. Then
  $B\in \mathrm{g}^{1}+\mathbb{C}I_{n+1}$ and hence
   $\widetilde{\mathrm{g}}\subset \mathrm{g}^{1}+\mathbb{C}I_{n+1}$. Since $\mathrm{g}^{1}\subset \widetilde{\mathrm{g}}$
    and   $\mathbb{C}I_{n+1}\subset \widetilde{\mathrm{g}}$, it follows that
    $\mathrm{g}^{1}+\mathbb{C}I_{n+1} \subset \widetilde{\mathrm{g}}$ (since $\widetilde{\mathrm{g}}$
    is an additive group, by Lemma ~\ref{LL:002}). This proves (i).
\
\\
$(ii)$ Since $\Psi(\mathfrak{q})=\mathrm{g}^{1}$ and $v_{0}=(1,
w_{0})$, we obtain for every
 $f=(B,b)\in \mathfrak{q}$, \begin{align*}\Psi(f)v_{0}& =\left[\begin{array}{cc}
                                                                          0 & 0 \\
                                                                          b & B
                                                                        \end{array}
\right]\left[\begin{array}{c}
               1 \\
               w_{0}
             \end{array}
\right]\\ \ & =\left[\begin{array}{c}
               0\\
               b+Bw_{0}
             \end{array}
\right]\\ \ & =\left[\begin{array}{c}
               0\\
               f(w_{0})
             \end{array}
\right].
\end{align*}
Hence $\mathrm{g}^{1}_{v_{0}}=\{0\}\times\mathfrak{q}_{w_{0}}$.
  \end{proof}
  \medskip

  \begin{lem}\label{Lkl:01} Under notations above, we have:\
   i) $\{0\}\times \mathbb{C}^{n}\oplus \mathbb{C}v_{0}=\mathbb{C}^{n+1}$.
   \ \\ ii)
   $p_{1}(\widetilde{\mathrm{g}}_{v_{0}})=\mathfrak{q}_{w_{0}}$,
   where $p_{1}: \ \{0\}\times \mathbb{C}^{n}\oplus
   \mathbb{C}v_{0}\longrightarrow\{0\}\times \mathbb{C}^{n}$
   setting by $p_{1}(v+\lambda v_{0})=v$ for every $v\in \{0\}\times \mathbb{C}^{n}\oplus
   \mathbb{C}v_{0}$.
   \end{lem}
   \medskip

   \begin{proof} The prrof of i) is obvious since
   $v_{0}=(1,w_{0})$. \ \\
   The proof of ii): By Lemma~\ref{L:01234}, ii), we have
   $\mathrm{g}^{1}_{v_{0}}=\{0\}\times\mathfrak{q}_{w_{0}}\subset \{0\}\times
   \mathbb{C}^{n}$. The proof follows then  by Lemma~\ref{L:01234}, i),
   because $\widetilde{\mathrm{g}}=\mathrm{g}^{1}+\mathbb{C}I_{n+1}$.
   \end{proof}
   \medskip

Denote by $V'=P(V)$ and $U=p_{1}(V)$. For such choise of the
matrix $P\in \Phi(GA(n+1, \C))$, we can write
$$P=\left[\begin{array}{cc}
  1 & 0 \\
  d & Q
\end{array}\right],\ \ \mathrm{with}\ \ \ Q\in GL(n, \C).$$  Set $h=(Q,d)\in GA(n, \C)$,
 $U'=h^{-1}(U)$ and $G^{*}=G\cap GL(n+1,\C)$. We have $G^{*}$ is an abelian semigroup of $GL(n+1,\C)$. Then:
\medskip

\begin{lem}\label{KKHJ:01}\cite{aA-Hm12} Under notations above, we have: \ \\ (i) every
orbit of $G^{*}$ and contained in $V'$ is minimal in it.\ \\ (ii)
if $G^{*}$ has a dense orbit in $V'$ then every orbit of $V'$ is
dense in $\mathbb{C}^{n+1}$.
   \end{lem}
   \medskip

   \begin{lem}\label{KjkJ:01}\cite{aA-Hm12} Under notations above,
   the following assertions are equivalent: \ \\
   (i) $G$ has a dense orbit.\ \\ (ii)
 $G^{*}$ has a dense orbit.
   \end{lem}
   \medskip

Denote by $\mathcal{G}^{*}=\mathcal{G}\cap GA(n, \C)$.

\begin{lem}\label{Klll:01} Under notations above, we have:\ \\
   (i) every orbit of $\mathcal{G}^{*}$ and contained in $U'$ is minimal in $U'$.\ \\
   (ii) if $\mathcal{G}$ has a dense orbit in $U'$ then every orbit
   of $U'$ is dense in $\mathbb{C}^{n}$.
   \end{lem}
   \medskip

   \begin{proof}(i) \ \\ (ii)
   \end{proof}
   \medskip

   \begin{lem}\label{ALLA:01} Under notations above. Let $x\in \mathbb{C}^{n}$. Then if $y\in J_{\mathcal{G}}(x)$ then
   $(1,y)\in J_{G}(1,x)$.
      \end{lem}
   \medskip

   \begin{proof} Let $x\in J_{\mathcal{G}}(y)$, then
   \end{proof}
   \medskip

\begin{lem}\label{LL0L:1} The following assertions are equivalent:\
\begin{itemize}
  \item [(i)] $\overline{\mathfrak{q}_{w_{0}}}=\mathbb{C}^{n}$.\
  \item [(ii)] $\overline{\mathrm{g}^{1}_{v_{0}}}=\{0\}\times\mathbb{C}^{n}$.\
  \item [(iii)] $\overline{\widetilde{\mathrm{g}}_{v_{0}}}=\mathbb{C}^{n+1}$.\
\end{itemize}
\end{lem}
\medskip

\begin{proof} $(i)\Longleftrightarrow(ii)$  follows from the fact that $\{0\}\times\mathfrak{q}_{w_{0}}=\mathrm{g}^{1}_{v_{0}}$
 (Lemma ~\ref{L:01234},(ii)).
\
\\
$(ii)\Longrightarrow(iii):$ By Lemma ~\ref{L:01234},(ii),
$\widetilde{\mathrm{g}}_{v_{0}}=\mathrm{g}^{1}_{v_{0}}+\mathbb{C}v_{0}$.
    Since $v_{0}=(1,w_{0})\notin \{0\}\times\mathbb{C}^{n}$ and
 $\mathbb{C}I_{n+1}\subset \widetilde{\mathrm{g}}$,
 we obtain $\mathbb{C}v_{0}\subset\widetilde{\mathrm{g}}_{v_{0}}$ and so
 $\mathbb{C}v_{0}\subset\overline{\widetilde{\mathrm{g}}_{v_{0}}}$. Therefore
 $\mathbb{C}^{n+1}=\{0\}\times\mathbb{C}^{n}\oplus \mathbb{C}v_{0}=
 \overline{\mathrm{g}^{1}_{v_{0}}}\oplus \mathbb{C}v_{0}\subset \overline{\widetilde{\mathrm{g}}_{v_{0}}}$
 (since, by Lemma~\ref{LL:002}, $\widetilde{\mathrm{g}}_{v_{0}}$ is an additive subgroup of $\mathbb{C}^{n+1}$).
 Thus $\overline{\widetilde{\mathrm{g}}_{v_{0}}}= \mathbb{C}^{n+1}$.\
 \\
$(iii)\Longrightarrow(ii):$ Let $x\in \mathbb{C}^{n}$, then
$(0,x)\in \overline{\widetilde{\mathrm{g}}_{v_{0}}}$
 and there exists a sequence $(A_{m})_{m\in\mathbb{N}}\subset\widetilde{\mathrm{g}}$ such that $\underset{m\to+\infty}{lim}A_{m}v_{0}=(0,x)$.
 By Lemma~\ref{L:01234}, we can write $A_{m}v_{0}=\lambda_{m}v_{0}+B_{m}v_{0}$ with $\lambda_{m}\in \mathbb{C}$ and
 $B_{m}=\left[\begin{array}{cc}
                0 & 0 \\
                b_{m} & B^{1}_{m}
              \end{array}
 \right]\in\mathrm{g}^{1}$ for every $m\in\mathbb{N}$. Since $B_{m}v_{0}\in\{0\}\times \mathbb{C}^{n}$ for every
 $m\in\mathbb{N}$ then $A_{m}v_{0}=(\lambda_{m},\ b_{m}+B^{1}_{m}w_{0}+\lambda_{m}w_{0})$.
It follows that $\underset{m\to+\infty}{lim}\lambda_{m}=0$ and
$\underset{m\to+\infty}{lim}A_{m}v_{0}=\underset{m\to+\infty}{lim}B_{m}v_{0}=(0,
x)$, thus $(0,x)\in\overline{\mathrm{g}^{1}_{v_{0}}}$. Hence
$\{0\}\times\mathbb{C}^{n}\subset
\overline{\mathrm{g}^{1}_{v_{0}}}$. Since
$\mathrm{g}^{1}\subset\Psi(MA(n, \mathbb{C}))$,
$\mathrm{g}^{1}_{v_{0}}\subset \{0\}\times\mathbb{C}^{n}$ then we
conclude that
$\overline{\mathrm{g}^{1}_{v_{0}}}=\{0\}\times\mathbb{C}^{n}$.
\end{proof}
\medskip

\begin{lem}\label{LL1L:9} Let $x\in\mathbb{C}^{n}$ and $G=\Phi(\mathcal{G})$. The following are equivalent:\
\begin{itemize}
  \item [(i)] $\overline{O_{\mathcal{G}}(x)}=\mathbb{C}^{n}$.
  \item [(ii)] $\overline{O_{G(1,x)}}=\{1\}\times\mathbb{C}^{n}$.
  \item [(iii)] $\overline{O_{\widetilde{G}}(1,x)}=\mathbb{C}^{n+1}$.
\end{itemize}
\end{lem}
\medskip

\begin{proof} $(i)\Longleftrightarrow (ii):$  is obvious since $\{1\}\times\mathcal{G}(x)=G(1,x)$ by construction.
\
\\
$(iii)\Longrightarrow (ii):$  Let $y\in \mathbb{C}^{n}$ and
$(B_{m})_{m}$ be a sequence in $\widetilde{G}$ such that
$\underset{m\to +\infty}{lim}B_{m}(1,x)=(1,y)$. One can write
$B_{m}=\lambda_{m}\Phi(f_{m})$, with $f_{m}\in \mathcal{G}$ and
$\lambda_{m}\in\mathbb{C}^{*}$, thus $B_{m}(1,x)=(\lambda_{m},\ \
\lambda_{m}f_{m}(x))$, so $\underset{m\to
+\infty}{lim}\lambda_{m}=1$. Therefore, $\underset{m\to
+\infty}{lim}\Phi(f_{m})(1,x)=\underset{m\to
+\infty}{lim}\frac{1}{\lambda_{m}}B_{m}(1,x)=(1,y)$. Hence,
$(1,y)\in \overline{O_{G}(1,x)}$. \
\\
$(ii)\Longrightarrow (iii):$ Since
$\mathbb{C}^{n+1}\backslash(\{0\}\times
\mathbb{C}^{n})=\underset{\lambda\in
\mathbb{C}^{*}}{\bigcup}\lambda
\left(\{1\}\times\mathbb{C}^{n}\right)$
 and for every $\lambda\in\mathbb{C}^{*}$, $\lambda G(1, x)\subset O_{\widetilde{G}}(1,x)$, we get \begin{align*}
\mathbb{C}^{n+1} & =
\overline{\mathbb{C}^{n+1}\backslash(\{0\}\times
\mathbb{C}^{n})}\\ \ & =\overline{\underset{\lambda\in
\mathbb{C}^{*}}{\bigcup}\lambda
\left(\{1\}\times\mathbb{C}^{n}\right)}\\ \ & =
\overline{\underset{\lambda\in \mathbb{C}^{*}}{\bigcup}\lambda
\overline{O_{G}(1,x)}} \subset \overline{O_{\widetilde{G}}(1,x)}
\end{align*}
Hence $\mathbb{C}^{n+1}=\overline{O_{\widetilde{G}}(1, x)}$.
\end{proof}

\medskip
\
\\
Recall the following result proved in \cite{aA-Hm12} which applied
to $G$ can be stated as following:
\medskip

\begin{prop}\label{p:10}$($\cite{aA-Hm12}, Proposition 5.1$)$ Let $G$ be an abelian sub-semigroup of $M_{n}(\mathbb{C})$ such that $G^{*}$ is generated by $A_{1},\dots,A_{p}$
and let  $B_{1},\dots,B_{p}\in \mathrm{g}$ such that  $A_{k} =
e^{B_{k}}$, $k = 1,\dots,p$ and $P\in GL(n+1,\mathbb{C})$
satisfying $P^{-1}GP\subset \mathcal{K}_{\eta,r}(\mathbb{C})$.
Then:
 $$\mathrm{g} =
\underset{k=1}{\overset{p}{\sum}}\mathbb{N}B_{k}+2i\pi\underset{k=1}{\overset{r}{\sum}}\mathbb{Z}PJ_{k}P^{-1}
\ \  \mathrm{and} \ \
 \ \mathrm{g}_{v_{0}} = \underset{k=1}{\overset{p}{\sum}}\mathbb{N}B_{k}v_{0} + \underset{k=1}{\overset{r}{\sum}} 2i\pi \mathbb{Z}Pe^{(k)},$$
where $J_{k} =\mathrm{diag}(J_{k,1},\dots, J_{k,r})$ with \
$J_{k,i}=0\in \mathbb{T}_{n_{i}}(\mathbb{C})$ if $i\neq k$ and
$J_{k,k} = I_{n_{k}}$.
 \end{prop}
\medskip

\begin{prop} \label{p:11} Let  $\mathcal{G}$ be an abelian sub-semigroup of  $GA(n, \mathbb{C})$ such that $\mathcal{G}^{*}$ is generated by
 $f_{1},\dots,f_{p}$ and let  $f'_{1},\dots,f'_{p}\in
\mathfrak{q}$ such that $e^{\Psi(f'_{k})}=\Phi(f_{k})$, $k =
1,..,p$.  Let $P$ be as in Proposition~\ref{p:2}. Then:
 $$\mathfrak{q}_{w_{0}}=\left\{\begin{array}{c}
                                     \underset{k=1}{\overset{p}{\sum}}\mathbb{N}f'_{k}(w_{0}) + \underset{k=2}{\overset{r}{\sum}} 2i\pi \mathbb{Z}p_{2}(Pe^{(k)}),\ if\ r\geq2 \\
                                     \underset{k=1}{\overset{p}{\sum}}\mathbb{N}f'_{k}(w_{0}) , \ if\ r=1\ \ \ \ \ \ \ \ \  \ \ \ \ \ \ \ \ \ \  \ \ \ \ \ \ \ \
                                   \end{array}
\right.$$
 \end{prop}

\begin{proof} Let  $G=\Phi(\mathcal{G})$.  Then  $G$ is
generated by  $\Phi(f_{1}),\dots,\Phi(f_{p})$.\ Apply Proposition
~\ref{p:10} to $G$, $A_{k}=\Phi(f_{k})$,
 $B_{k}=\Psi(f'_{k})\in\mathrm{g}^{1}$,  then  we have
$$\mathrm{g} =
\underset{k=1}{\overset{p}{\sum}}\mathbb{Z}\Psi(f'_{k}) +2i\pi
\mathbb{Z} \underset{k=1}{\overset{r}{\sum}} PJ_{k}P^{-1}.$$

We have
$\underset{k=1}{\overset{p}{\sum}}\mathbb{Z}\Psi(f'_{k})\subset\Psi(MA(n,
\mathbb{C}))$. Moreover, for every $k=2,\dots,r$, $J_{k}\in
\Psi(MA(n, \mathbb{C}))$, hence $PJ_{k}P^{-1}\in
\Psi(MA(n,\mathbb{C}))$, since $P\in \Phi(GA(n, \mathbb{C}))$.
 However,  $mPJ_{1}P^{-1}\notin \Psi(MA(n,\mathbb{C}))$ for every $m\in\mathbb{Z}\backslash\{0\}$, since $J_{1}$ has
  the form $J_{1}=\mathrm{diag}(1, J')$ where $J'\in M_{n}(\mathbb{C})$. As  $\mathrm{g}^{1}=\mathrm{g}\cap \Psi(MA(n, \mathbb{C}))$, then $mPJ_{1}P^{-1}\notin \mathrm{g}^{1}$ for every $m\in\mathbb{Z}\backslash\{0\}$.
 Hence we obtain:
$$\mathrm{g}^{1} =\left\{\begin{array}{c}
                       \underset{k=1}{\overset{p}{\sum}}\mathbb{N}\Psi(f'_{k}) +
\underset{k=2}{\overset{r}{\sum}} 2i\pi \mathbb{Z}PJ_{k}P^{-1}, \
if\ r\geq2 \ \ \  \\
                       \underset{k=1}{\overset{p}{\sum}}\mathbb{N}\Psi(f'_{k}), \ if\ r=1 \ \ \ \ \ \ \ \ \ \ \ \ \ \ \ \ \ \ \ \ \ \ \ \ \ \ \ \
                     \end{array}
\right.$$

Since $J_{k}u_{0}=e^{(k)}$, we get $$\mathrm{g}^{1}_{v_{0}}
=\left\{\begin{array}{c}
                       \underset{k=1}{\overset{p}{\sum}}\mathbb{N}\Psi(f'_{k})v_{0} +
\underset{k=2}{\overset{r}{\sum}} 2i\pi \mathbb{Z}Pe^{(k)}, \ if\
r\geq2 \ \ \ \ \\
                       \underset{k=1}{\overset{p}{\sum}}\mathbb{N}\Psi(f'_{k})v_{0}, \ if\ r=1 \ \ \ \ \ \ \ \ \ \ \ \ \ \ \ \ \ \ \ \ \ \ \ \ \ \ \ \
                     \end{array}
\right.$$

By Lemma~\ref{L:01234},(iii), one has
$\{0\}\times\mathfrak{q}_{w_{0}}=\mathrm{g}^{1}_{v_{0}}$ and
$\Psi(f'_{k})v_{0}=(0,\ f'_{k}(w_{0}))$, so
$\mathfrak{q}_{w_{0}}=p_{2}\left(\mathrm{g}^{1}_{v_{0}}\right)$.
It follows that $$\mathfrak{q}_{w_{0}}=\left\{\begin{array}{c}
                                     \underset{k=1}{\overset{p}{\sum}}\mathbb{N}f'_{k}(w_{0}) + \underset{k=2}{\overset{r}{\sum}} 2i\pi \mathbb{Z}p_{2}(Pe^{(k)}),\ if\ r\geq2 \\
                                     \underset{k=1}{\overset{p}{\sum}}\mathbb{N}f'_{k}(w_{0}) , \ if\ r=1\ \ \ \ \ \ \ \ \  \ \ \ \ \ \ \ \ \ \  \ \ \ \ \ \ \ \
                                   \end{array}
\right.$$ The proof is completed.
\end{proof}
\medskip

Recall the following proposition which was proven in \cite{mW}:
\\

\begin{prop}\label{p:12}$(cf.$ \cite{mW}, $page$  $35)$. Let $F = \mathbb{Z}u_{1}+\dots+\mathbb{Z}u_{m}$ with $u_{k} =
Re(u_{k})+i Im(u_{k})$, where $\mathrm{Re}(u_{k})$,
$\mathrm{Im}(u_{k})\in \mathbb{R}^{n}$, $k = 1,\dots, m$. Then $F$
is dense in $\mathbb{C}^{n}$ if and only if for every
$(s_{1},\dots,s_{m})\in
 \mathbb{Z}^{p}\backslash\{0\}$ :
$$\mathrm{rank}\left[\begin{array}{cccc }
 \mathrm{Re}(u_{1 }) &\dots &\dots & \mathrm{Re}(u_{m }) \\
   \mathrm{Im}(u_{1 }) &\dots &\dots & \mathrm{Im}(u_{m}) \\
  s_{1 } &\dots&\dots & s_{m }
 \end{array}\right] =\ 2n+1.$$
\end{prop}
\bigskip

\
\\
{\it Proof of Theorem ~\ref{T:00}.}  Suppose that $\mathcal{G}$ is
hypercyclic, so $\overline{\mathcal{G}(x)}=\mathbb{C}^{n}$ for
some $x\in\mathbb{C}^{n}$. By Lemma~\ref{LL1L:9},(iii),
$\overline{\widetilde{G}(1,x)}=\mathbb{C}^{n+1}$ and by Theorem
~\ref{T:5}, $\overline{\widetilde{G}(v_{0})}=\mathbb{C}^{n+1}$ and
so $\overline{\widetilde{\mathrm{g}}(v_{0})}=\mathbb{C}^{n+1}$.
Then
$p_{1}(\overline{\widetilde{\mathrm{g}}(v_{0})})=\mathbb{C}^{n}$,
where  where $p_{1}: \ \{0\}\times \mathbb{C}^{n}\oplus
   \mathbb{C}v_{0}\longrightarrow\{0\}\times \mathbb{C}^{n}$
   setting by $p_{1}(v+\lambda v_{0})=v$ for every $v\in \{0\}\times \mathbb{C}^{n}\oplus
   \mathbb{C}v_{0}$. By Lemma~\ref{Lkl:01},
   $p_{1}(\widetilde{\mathrm{g}}(v_{0}))= \mathfrak{q}_{w_{0}}$.
   Therefore
$\overline{\mathfrak{q}_{w_{0}}}= \mathbb{C}^{n}$. By
Proposition~\ref{p:11},
$$\mathfrak{q}_{w_{0}}=\left\{\begin{array}{c}
                                     \underset{k=1}{\overset{p}{\sum}}\mathbb{N}f'_{k}(w_{0}) + \underset{k=2}{\overset{r}{\sum}} 2i\pi \mathbb{Z}p_{2}(Pe^{(k)}),\ if\ r\geq2 \\
                                     \underset{k=1}{\overset{p}{\sum}}\mathbb{N}f'_{k}(w_{0}) , \ if\ r=1\ \ \ \ \ \ \ \ \  \ \ \ \ \ \ \ \ \ \  \ \ \ \ \ \ \ \
                                   \end{array}
\right.$$

Since $m:=p+r-1\leq 2n$ then by Lemma~\ref{p:12},
$\mathfrak{q}_{w_{0}}$ can not be dense in $\mathbb{C}^{n}$, a
contradiction. $\hfill{\Box}$

\bigskip

\section{{\bf Proof of Theorem~\ref{T:1}}}
Recall that $\mathcal{G}$ is an abelian sub-semigroup of $MA(n,
\mathbb{C})$, $G=\Phi(\mathcal{G})$ and $\widetilde{G}$ is the
semi-group generated by $G$ and $\mathbb{C}I_{n+1}=\{\lambda
I_{n+1}:\  \  \ \lambda\in \mathbb{C} \}$. Then $\widetilde{G}$ is
an abelian sub-semigroup of $M_{n+1}(\mathbb{C})$. Let $k\geq 1$
be an integer, denote by $(\mathbb{C}^{n+1})^{k}$ the $k$-fold
Cartesian product of $\mathbb{C}^{n+1}$. For every $u = (x_{1},
\dots, x_{k})\in (\mathbb{C}^{n+1})^{k}$, the orbit of $u$ under
the action of $G$ (resp. $\widetilde{G}$) on
$(\mathbb{C}^{n+1})^{k}$ is denoted $$O^{k}_{G}(u) =
\{(Ax_{1},\dots ,Ax_{k}) : A \in G\}.$$ $$(\mathrm{resp}.\ \
O^{k}_{\widetilde{G}}(u) = \{(Ax_{1},\dots ,Ax_{k}) : A \in
\widetilde{G}\}.)$$

\medskip

\begin{lem}\label{L:00120001} Under above notation. Let $k\geq 2$ and $u=(x_{1},\dots, x_{n})
\in (\mathbb{C}^{n})^{k}$. Denote by
$\widetilde{u}:=((1,x_{1}),\dots, (1,x_{n}))$ then the following
assertions are equivalent:\ \\ (i)
$\overline{O^{k}_{\mathcal{G}}(u)}=(\mathbb{C}^{n})^{k}$.\ \\ (ii)
$\overline{O^{k}_{G}(\widetilde{u})}=\left(\{1\}\times\mathbb{C}^{n}\right)^{k}$.
\
\\ (iii)
$\overline{O^{k}_{\widetilde{G}}(\widetilde{u})}=(\mathbb{C}^{n+1})^{k}$.
\end{lem}

\medskip

\begin{proof}$(i)\Longrightarrow (ii)$:  Let $v=(y_{1},\dots,
y_{k})\in(\mathbb{C}^{n})^{k}$, then $v\in
\overline{O^{k}_{\mathcal{G}}(u)}$, so there exists a sequence
$(g_{m})_{m}$ in $\mathcal{G}$ such that $\underset{m\to
+\infty}{lim}(g_{m}(x_{1}),\dots,g_{m}(x_{k}))=(y_{1},\dots,
y_{k})$. Then $\underset{m\to
+\infty}{lim}(\Phi(g_{m})(1,x_{1}),\dots,\Phi(g_{m})(1,x_{k}))=((1,y_{1}),\dots,
(1,y_{k}))$. Therefore $\widetilde{v}:=((1,y_{1}),\dots,
(1,y_{k}))\in \overline{O^{k}_{G}(\widetilde{u})}$, since
$\Phi(f_{m})\in G$ for every $m$.\
\\ $(ii)\Longrightarrow (iii)$: Let $v=(y_{1},\dots,
y_{k})\in(\mathbb{C}^{n})^{k}$ and $(B_{m})_{m}$ be a sequence in
$\widetilde{G}$ such that $\underset{m\to
+\infty}{lim}(B_{m}(1,x_{1}),\dots,
B_{m}(1,x_{k}))=\widetilde{v}$. One can write
$B_{m}=\lambda_{m}\Phi(f_{m})$, with $f_{m}\in \mathcal{G}$ and
$\lambda_{m}\in\mathbb{C}^{*}$, thus
$B_{m}(1,x_{j})=(\lambda_{m},\ \ \lambda_{m}f_{m}(x_{j}))$ for
every $j=1,\dots, k$, so $\underset{m\to
+\infty}{lim}\lambda_{m}=1$ because  $\underset{m\to
+\infty}{lim}B_{m}(1, x_{1})=(1,y_{1})$. Therefore,
$\underset{m\to +\infty}{lim}\Phi(f_{m})(1,x_{j})=\underset{m\to
+\infty}{lim}\frac{1}{\lambda_{m}}B_{m}(1,x_{j})=(1,y_{j})$, hence
\begin{align*}
\widetilde{v}& =
\underset{m\to+\infty}{lim}\left(\frac{1}{\lambda_{m}}B_{m}(1,x_{1}),\dots,\frac{1}{\lambda_{m}}B_{m}(1,x_{k})\right)\\
\ &=
\underset{m\to+\infty}{lim}\left(\Phi(f_{m})(1,x_{1}),\dots,\Phi(f_{m})(1,x_{k})\right)
\end{align*}
 It follows that $\widetilde{v}\in
\overline{O^{k}_{G}(\widetilde{u})}$.\ \\
$(iii)\Longrightarrow(i):$  Let $v=(y_{1},\dots,
y_{k})\in(\mathbb{C}^{n})^{k}$, then $\widetilde{v}\in
\overline{O^{k}_{\widetilde{G}}(\widetilde{u})}$, so  there exists
a sequence $(B_{m})_{m}$ in $\widetilde{G}$ such that
$\underset{m\to
+\infty}{lim}(B_{m}(1,x_{1}),\dots,B_{m}(1,x_{k}))=((1,y_{1}),\dots,
(1,y_{k}))$.  One can write $B_{m}=\lambda_{m}\Phi(f_{m})$, with
$f_{m}\in \mathcal{G}$ and $\lambda_{m}\in\mathbb{C}^{*}$, thus
$B_{m}(1,x_{j})=(\lambda_{m},\ \ \lambda_{m}f_{m}(x_{j}))$ for
every $j=1,\dots, k$, so
$$\underset{m\to+\infty}{lim}\lambda_{m}=1.$$ Therefore,
$\underset{m\to +\infty}{lim}\Phi(f_{m})(1,x_{j})=\underset{m\to
+\infty}{lim}\frac{1}{\lambda_{m}}B_{m}(1,x_{j})=(1,y_{j})$.
Therefore $v\in \overline{O^{k}_{\mathcal{G}}(u)}.$\
\end{proof}
\medskip

\begin{prop}\label{P:P10}$($\cite{AA}, Theorem 1.2$)$ Let $L$ be an abelian semigroup generated by $p$ matrices
($p\geq 1$) on $\mathbb{K}^{n}$ ($\mathbb{K} = \mathbb{R}$ or
$\mathbb{C}$). Then the action of $L$ on $\mathbb{K}^{n}$ is never
k-transitive for $k\geq  2$.
\end{prop}
\medskip

\begin{proof}[Proof of Theorems ~\ref{T:1}]  The proof results directly from Lemma ~\ref{L:00120001} and
Proposition~\ref{P:P10}.
\end{proof}
\bigskip

\bibliographystyle{amsplain}
\vskip 0,4 cm

\end{document}